\documentclass[12pt,a4paper]{article}
\usepackage{amsmath,amssymb, amsthm,latexsym,color,epsfig,a4, enumerate}
\usepackage{soul}
\usepackage{bbm}
\usepackage{a4wide} 

\newtheorem{thm}{Theorem}[section]
\newtheorem{lem}[thm]{Lemma}
\newtheorem{claim}[thm]{Claim}

\newtheorem{cor}[thm]{Corollary}

\newcommand{\cS}{\mathcal{S}}
\newcommand{\cF}{\mathcal{F}}

\begin{document}
\title{On the WalkerMaker--WalkerBreaker games}
\author{Jovana Forcan\thanks{Department of Mathematics and Informatics, Faculty of Sciences, University of Novi Sad, Serbia. 
\newline
Department of Mathematics, Informatics and Physics, Faculty of Philosophy, University of East Sarajevo, Bosnia and Herzegovina. 
Email: dmi.jovana.jankovic@student.pmf.uns.ac.rs}
 \quad Mirjana Mikala\v{c}ki\thanks{Department of Mathematics and Informatics, Faculty of Sciences, University of Novi Sad, Serbia.
Email: mirjana.mikalacki@dmi.uns.ac.rs.}}

\maketitle

\begin{abstract}
We study the unbiased WalkerMaker--WalkerBreaker games on the edge set of the complete graph on $n$ vertices, $K_n$, a variant of well-known Maker--Breaker positional games, where both players have the restriction on the way of playing. Namely, each player has to choose her/his edges according to a walk. Here, we focus on two standard graph games -- the Connectivity game and the Hamilton cycle game and show how quickly WalkerMaker can win both games. 
\end{abstract}

\maketitle
\section{Introduction}
In this paper we study a variant of the well-known Maker--Breaker  games. Let $X$ be a finite set and $\cF$ a family of subsets of $X$. Given two positive integers $a$ and $b$, in the $(a:b)$ Maker--Breaker positional game $(X,\cF)$, two players, Maker and Breaker, take turns in claiming $a$, respectively $b$, elements of $X$ until all elements are claimed. The basic setup is that both players claim exactly \textit{one} element per turn, i.e.\ $a=b=1$. These games are called \textit{unbiased}. Maker wins the game if she claims all the elements of some $F\in \cF$ by the end of the game. Breaker wins otherwise. No draw is possible. The set $X$ is often referred to as the \textit{board} and $\cF$ as the family of \textit{winning sets}. Maker--Breaker games have been studied a lot in the last $30$ years and more about this type of positional games and others can be found in the book of Beck~\cite{BeckBook} and in the recent monograph of Hefetz, Krivelevich, Stojakovi\'{c} and Szab\'{o}~\cite{HKSSbook}. 

It is very natural to play Maker--Breaker games on the edges of the given graph $G$, i.e.\ when $X=E(G)$, and the winning sets are some standard graph theoretic structures, such as spanning trees, Hamilton cycles, perfect matchings, triangles etc. In this paper we will focus on the games on the edges of the complete graph on $n$ vertices, where $X=E(K_n)$ and in particular, we are interested in two standard games: the \textit{Connectivity game}, where the winning sets are the edge sets of all spanning trees of $K_n$ and the \textit{Hamilton cycle game}, where the winning sets are the edge sets of all Hamilton cycles of $K_n$. Maker can win the Connectivity game in $n-1$ moves, as showed by Lehman in~\cite{Lehman}. In the Hamilton cycle game Maker needs to make at least $n+1$ moves to win, as showed by Hefetz et al.\ in~\cite{HKSS2}, and Hefetz and Stich in~\cite{HS09} obtained that $n+1$ moves is enough for her to win.

As it turns out, Maker can easily win in most of the standard unbiased graph games on $E(K_n)$, for sufficiently large $n$. In order to give Breaker more power, several approaches were introduced over the years. One of them is to play the \textit{biased} $(1:b)$ games for $b\geq 1$, as suggested by Chv\'{a}tal and Erd\H{o}s in~\cite{CE}. Another one is to reduce the number of winning sets by making the base graph sparser and play on the random board, as proposed by Stojakovi\'{c} and Szab\'{o} in~\cite{SS05}. 

Recently, Espig, Frieze, Krivelevich and Pegden~\cite{EFKP15} introduced the new approach to make up for Maker's advantage in the unbiased games. In the \textit{Walker--Breaker} games, Walker, having the role of Maker, has to claim her edges according to a \textit{walk}, while Breaker has no restrictions on the way he moves. More specifically, Walker can choose any vertex as her first position. 
When she is positioned at vertex $v$ and it is her turn to play, she can only claim an edge incident with $v$ not previously claimed by Breaker. The other endpoint of the claimed edge becomes her new position. 

Due to their recent appearance, little is known about Walker--Breaker games (see~\cite{CT16,EFKP15}) and lots of questions are still open. For Walker, even in an unbiased $(1:1)$ Walker--Breaker game, it is impossible to create a spanning structure. The longest path that she can create in the $(1:1)$ game on $E(K_n)$ has $n-2$ vertices, as shown in~\cite{EFKP15}. In the same paper, the authors asked the following question.

\paragraph{Question~\cite{EFKP15}:} \textit{What happens if Breaker is also a walker?}\\

In this paper we address this question and study the unbiased $(1:1)$ WalkerMaker--WalkerBreaker games (\textit{WMaker--WBreaker} games for brevity), in which each player has to claim her/his edges according to a walk, i.e.\ when a player is at some vertex $v$, she/he can only choose edges incident with $v$ not previously claimed by the opponent. With this restriction, we look at the two standard games -- the Connectivity game and the Hamilton cycle game on $E(K_n)$, for sufficiently large $n$. We show that situation changes when both players are walkers and in that case it is possible for WMaker to create both a spanning tree and a Hamilton cycle. Moreover, WMaker can do it without wasting too many moves. In particular, we obtain the following.

\begin{thm} \label{teorema1}
In the $(1:1)$ WMaker--WBreaker Connectivity game on $E(K_n)$, WMaker has a strategy to win in at most $n+1$ moves. 
\end{thm}

\begin{thm} \label{teorema2}
In the $(1:1)$ WMaker--WBreaker Hamilton cycle game on $E(K_n)$, WMaker has a strategy to win in at most $n+6$ moves. 
\end{thm}

We also look at WBreaker's possibilities to postpone WMaker's win in the Connectivity game.

\begin{thm}
\label{breaker}
In the $(1:1)$ WMaker--WBreaker Connectivity game on $E(K_n)$, WBreaker, as the second player, has a strategy to postpone WMaker's win by at least $n$ moves. 
\end{thm}

The rest of paper is organized as follows. In Section~\ref{mpf}, we prove Theorems~\ref{teorema1} and~\ref{teorema2}. In Section~\ref{brP} we prove Theorem~\ref{breaker}. Finally, in Section~\ref{sec::last} we give some concluding remarks.
\subsection{Notation}

Our graph-theoretic notation is standard and follows that of~\cite{West}. In particular, throughout the paper we use the following. 

Given a graph $G$, $V(G)$ and $E(G)$ denote its sets of vertices, respectively edges, and $v(G) = |V(G)|$ and $e(G) = |E(G)|$. Given two vertices $x,y\in V(G)$ an edge in $G$ is denoted by $xy$. Given a vertex $x\in V(G)$, we use $d_G(x)$ to denote the degree of vertex $x$ in $G$. For a set $A\subseteq V(G)$ and $x\in V(G)\setminus A$, let $d_G(x,A)$ denote the degree of $x$ towards $A$.

Assume that a WMaker--WBreaker game on the edge set of a given graph $G$ is in progress. At any point of the game, let $M$ and $B$ denote the graphs spanned by edges WMaker, respectively WBreaker, claimed so far. 

For some vertex $v$ we say that it is \textit{visited} by a player if he/she has claimed \textit{at least one} edge incident with $v$. A vertex is \textit{isolated/unvisited} if no edge incident to it is claimed. 
We use $U$ to denote the set of vertices that are still unvisited by WMaker, i.e.\ $U=V(G\setminus M)$. The edges in $E(G \setminus(M\cup B))$ are called \textit{free}.

Unless otherwise stated, we assume that WBreaker starts the game, i.e.\  one \textit{round} in the game consists of a move by WBreaker followed by a move of WMaker.

\section{Proofs of Theorems~\ref{teorema1} and~\ref{teorema2}}
\label{mpf}
\noindent First we define the strategy $\cS$, which WMaker will use in one part of both of the games in order to win.
\paragraph{Strategy $\cS$.} For her starting vertex, WMaker chooses the vertex $v_1$, in which WBreaker has finished his first move, and claims an edge $v_1u$ such that $d_B(u)=0$ (ties broken arbitrarily). 
In every other round WMaker checks if there exists an edge $e\in E(B)$, $e = pq$, s.t. $p, q \in U$, and from her current position $w$ claims $wp$, or $wq$, whichever is free. If both $wp$ and $wq$ are free she chooses $wp$ if $d_B(p)>d_B(q)$, and $wq$, if $d_B(q)>d_B(p)$ (ties are broken arbitrarily). If no such edge exists, WMaker 
from her current position $w$ claims a free edge $wu$ such that $u\in U$ and $d_B(u) = \mathrm{max} \{d_B(v): v\in U \}$, ties broken arbitrarily, for as long as $|U|\geq 3$. If all free edges $wu$ are such that $d_B(u)=0$ for all $u\in U$, then WMaker claims an arbitrary free edge $wu$. 
\\ \\
If WMaker plays according to the strategy $\cS$, then the following statements can be proven, which will be used in proving Theorem \ref{teorema1} and Theorem \ref{teorema2}. 
\begin{lem} \label{lema1}
In the $(1:1)$ WMaker--WBreaker game on $E(K_n)$ strategy $\cS$ guarantees WMaker that, as long as $|U|>2$, after each round, every WBreaker's edge is incident with some vertex $v\in V(M)$. 
\end{lem}
\begin{proof} The proof goes by induction on the number of rounds $k$. In the first round, we know that WMaker for her starting position chooses the vertex $v_1$ in which WBreaker has finished his first move and claims the edge $v_1u$, $u\in U, d_B(u)=0$. So, the statement holds after the first round. Assume that the statement is true after $k \geq 2$ rounds. 
Assume that in round $k+1$ WBreaker claimed edge $b_1b_2$ such that $b_1,b_2 \notin V(M)$. Denote WMaker's current position by $w$ and suppose that WMaker is not able to visit either $b_1$ or $b_2$ in round $k+1$. It follows that $wb_1, wb_2 \in E(B)$. Suppose that WBreaker claimed these edges in the following order: $b_1w, wb_2, b_2b_1$. This means that, in round in which WBreaker claimed edge $b_1w$, WMaker moved to some vertex different from $b_1$ and $w$, and after that round $b_1w$ was WBreaker's edge not incident with $V(M)$, which is a contradiction to the induction hypothesis.    
\end{proof}
\begin{cor} \label{c1}
In the $(1:1)$ WMaker--WBreaker game on $E(K_n)$, as long as $|U|>2$, strategy $\cS$ guarantees that after each round WMaker is positioned at some vertex $w$ such that $d_B(w,U)\leq 1$. 
\end{cor}
\begin{proof}
This already holds after the first round. Suppose that after some round $i>1$ WMaker is at vertex $w$ such that $d_B(w,U)=2$, that is $wu, wu' \in E(B)$, for some $u,u' \in U$ and suppose that WBreaker claimed $wu$ before $wu'$. Assume also that it is again WBreaker's turn and he claims some edge incident with $u' \in U$ in round $i+1$. \\
This contradicts Lemma \ref{lema1}, because when WBreaker claimed $wu$ in round $i-1$, WMaker visited some vertex different from $w$ and $u$, and after that round, $wu$ was WBreaker's edge not incident with $V(M)$. 
It follows that vertex in which WMaker is positioned at the end of each round can have degree $d_B(w,U)\leq 1$.  
\end{proof} 
\begin{cor} \label{cc1}
In the $(1:1)$ WMaker--WBreaker game on $E(K_n)$, the strategy $\cS$ guarantees WMaker that, as long as $|U|\geq 2$, after  WBreaker's move (and before WMaker's move) in some round $i\geq 2$, vertex $w$ in which WMaker finished her previous move, can have degree $d_B(w,U) \leq 2$. If $d_B(w,U)=2$, then WBreaker finished his move in round $i-1$ at vertex $w$.
\end{cor} 
\begin{cor} \label{c2}
In the $(1:1)$ WMaker--WBreaker game on $E(K_n)$, WMaker can build a path $P$ of length $n-3$ (with $n-2$ vertices) in $n-3$ moves by playing according to strategy $\cS$. 
\end{cor}
\begin{proof} Suppose that WMaker already built a path $P$ of length $n-4$ ($v(P)=n-3$). Let $U= \{u_1, u_2,u_3 \}$. Let vertex $w$ be WMaker's current position. If $wu_i \in E(B)$ for every $i\in\{1,2,3\}$, this means that after WBreaker's move in this round we have $d_B(w,U)=3$ and this contradicts Corollary~\ref{cc1}. 
\end{proof}
\begin{lem} \label{lema2}
In the $(1:1)$ WMaker--WBreaker game on $E(K_n)$, strategy $\cS$ guarantees WMaker that, as long as $|U|>2$, after each round there can be at most 2 vertices from $U$ belonging to $V(B)$.
\end{lem}
\begin{proof} The proof goes by induction on the number of rounds $k$. After the first round, there is only one vertex from $U$ visited by WBreaker. This is the vertex which WBreaker chose for his starting position. Suppose that after $k$ rounds, where $k>1$, there were at most two vertices from $U$ visited by WBreaker. Assume that WBreaker played his move in round $k+1$ and now it is WMaker's turn to play her move in  round $k+1$. If there are three vertices from $U$ visited by WBreaker, then by induction hypothesis, WBreaker touched one of these vertices in his last move (round $k+1$). Denote these vertices by $u_1,u_2,u_3$ and let $u_3$ be a vertex touched by WBreaker in his last move.
If WMaker is not able to claim any of edges $wu_1,wu_2,wu_3$ from her current position $w$, this means that WBreaker finished his $k^{\mathrm{th}}$ move also in the vertex $w$ and after round $k$ we had $u_1w,wu_2 \in E(B)$. For this, WBreaker needed three moves, which means that in round $k-2$ WBreaker claimed $wu_i$, for some $i\in\{1,2\}$. Since WMaker visited $w$ in round $k$, we get a contradiction to Lemma \ref{lema1} after round $k-2$, and also a contradiction to induction hypothesis after round $k-1$ because there were three vertices from $U$ visited by WBreaker. 
\end{proof}
\begin{lem} \label{lema3}
In the $(1:1)$ WMaker--WBreaker game on $E(K_n)$, strategy $\cS$ guarantees WMaker that for every vertex $x \in U$, $d_B(x)\leq 6$ holds at the moment when WMaker visits it for the first time.
\end{lem}
\begin{proof} Assume that WBreaker touched vertex $x$ for the first time in some round $i$ using the edge $ax$. We will show that $d_B(x)\leq 6$ at the moment WMaker visits it.
We analyse the following cases:
\paragraph{Case 1.} WMaker was already positioned at vertex $a$ at the beginning of round $i$, and after WBreaker claimed $ax$, there can be at most 2 additional vertices $u_1,u_2$ from $U$ visited by WBreaker before WMaker's move in round $i$, according to Lemma \ref{lema2}. \\
In our analysis, we will assume that both vertices $u_1,u_2 \in U$ are in $V(B)$.
When only one of vertices $u_1,u_2 \in U$ is in $V(B)$ or none of them belong to $V(B)$, analysis is similar, but much simpler.  
\paragraph{Case 2.} WMaker's current position is at some vertex $w$ and she visits $a$ for the first time in round $i$. Beside the vertex $x \in U$ there can be at most one vertex from $U$ visited by WBreaker after round $i$, according to Lemma \ref{lema2}. Denote this vertex with $u'$. 
In our analysis we will assume that there is such vertex $u' \in U$ which belongs to $V(B)$. Otherwise, analysis is similar, but much simpler. 
\paragraph{Case 3.} WMaker is at some vertex $w \neq a$ at the beginning of round $i$ and $a \in V(M)$. There can be at most 2 additional vertices $u_1,u_2 \in U \cap V(B)$ before WMaker's move in round $i$, according to Lemma \ref{lema2}. \\
In our analysis, we will assume that both vertices $u_1,u_2 \in U$ are in $V(B)$.
When only one of vertices $u_1,u_2 \in U$ is in $V(B)$ the analysis is similar, but much simpler. If none of these two vertices belong to $V(B)$, WMaker moves from $w$ to $x$ in round $i$, which completes the analysis. \\ 

In the following we analyse all three cases separately. 
\paragraph{Case 1.} By Corollary~\ref{cc1} it is not possible that all three edges $ax,au_1,au_2$ are in $E(B)$, and so WMaker can move to some $u_i$, $i\in\{1,2\}$, say $u_1$.\\ 
Then, if WBreaker claims the edge $xb$, for some $b\in U, b\notin \{u_1,u_2\}$, WMaker must move from $u_1$ to $b$ or $x$.
Since $d_B(x)=2$ and $d_B(b)=1$, the strategy $\cS$ will tell her to choose the edge $u_1x$. This is possible for WMaker to claim as  $u_1x, u_1b \notin E(B)$ (otherwise it would contradict Lemma~\ref{lema1} before round $i$).  \\ 
If WBreaker chose edge $xu_1$ in round $i+1$, WMaker moves from $u_1$ to $u_2$, and in round $i+2$ claims $u_2x$.  \\ \\
If WBreaker chose edge $xu_2$ in round $i+1$, then, if $d_B(x)> d_B(u_2)$, following $\cS$, WMaker moves to $x$. Otherwise, she moves to $u_2$ (suppose she moves to $u_2$ even if $d_B(u_2)= d_B(x)$). 
In round $i+2$, WMaker claims $u_2m_1$ for some $m_1\in U$. If in round $i+3$, WMaker is not able to claim $m_1x$, this means that:
\begin{enumerate}
\item[i.] WBreaker returned to $x$ along $u_2x$ in round $i+2$ and in round $i+3$ he claimed $xm_1$. Then, WMaker moves from $m_1$ to some $m_2\in U$ in round $i+3$. So, in round $i+4$, she will be able to claim $m_2x$. In that moment we would have $d_B(x)= 3$ because in this round WBreaker could have either returned to $x$ along $m_1x$ or claimed $m_1v$, for some $v\in V(K_n), v\neq x$ (and $d_B(v)<d_B(x)$). The edge $m_2x$ is free in the moment when WMaker wants to claim it. Otherwise, we will have a contradiction to Lemma \ref{lema1} before round $i+3$. 
\item[ii.] WBreaker claimed edges $u_2y_1$ and $y_1y_2$ for some $y_1,y_2 \in U$, in rounds $i+2$ and $i+3$, respectively. Since $y_1y_2 \in E(B)$ is not incident with $V(M)$, WMaker must visit $y_1$ or $y_2$ from $m_1$ in round $i+3$. Lemma~\ref{lema1} implies that edges $y_1m_1,y_2m_1 \notin E(B)$. Since $d_B(y_1)>d_B(y_2)$, WMaker moves to $y_1$. 
If WBreaker moves to some vertex $ v \notin U$ or to $v=x$, WMaker can claim $y_1x$ (and $d_B(x)\leq 3$). Otherwise, if WBreaker claims $y_2y_3$ for some $y_3 \in U$, in round $i+4$, strategy $\cS$ will tell WMaker to claim $y_1y_3$, because edge $y_2y_3$ is not incident with $V(M)$. 
If WBreaker claims $y_3v'$ for some $v' \neq x$, WMaker visits $x$ along the edge $y_3x$. Otherwise, if WBreaker claims $y_3x$, WMaker claims $y_3m_2$, for some $m_2\in U$. At that point $d_B(x)=3$.
\\
If $d_B(x)=4$ in round $i+6$, when WBreaker claimed $xm_2$, WMaker moves from $m_2$ to $y_2$, and afterwards, in round $i+7$, she moves from $y_2$ to $x$, where $d_B(x)=4$. Since WBreaker finished his move in round $i+6$ at vertex $m_2$, he is not able to prevent WMaker from visiting $x$ in round $i+7$. If WBreaker claimed $xu$, $u\neq m_2 $, in round $i+6$, WMaker can visit $x$ in this round by moving along the edge $m_2x$. 
\end{enumerate}
\paragraph{Case 2.} WMaker is at vertex $a$ at the beginning of round $i+1$. If WBreaker claims $xb$, where $b\in V(M)$, then WMaker moves to $u' \in U$ and in the following round claims $u'x$. This edge is free, otherwise we will have a contradiction to Lemma~\ref{lema1} before round $i$. \\ 
From now on, suppose that WBreaker claimed $xb$, where $b\in U$ in round $i+1$. By $\cS$, WMaker must claim the edge $ab$. Also, Lemma~\ref{lema1} implies that $ab \notin E(B)$.\\ 
Consider the following situations:
\begin{enumerate}
\item[i.] WBreaker claims edges $bc$ and $cx$, $c\in U$, in rounds $i+2$ and $i+3$, respectively. 
\begin{enumerate}
\item[a)] Let $c \neq u'$. Then WMaker claims $bu'$, $u'\in U$ in round $i+2$. In the following round, WMaker is able to claim $u'x$ or $u'c$ (due to Lemma~\ref{lema1} before round $i$), but she will move to vertex $x$, because $3=d_B(x)>d_B(c)=2$.\\ 
In case $b=u'$, WMaker first moves from $b$ to some $m_1\in U $ in round $i+2$, and then visits $x$ along the edge $m_1x$ in round $i+3$.
\item[b)]  If $c = u'$, then WMaker moves from $b$ to some $m_1 \in U$ in round $i+2$. Since $d_B(c) \geq d_B(x)$, after WBreaker claims $cx$ in round $i+3$, suppose that WMaker claims $m_1c$ (even if $d_B(c)=d_B(x)$, as otherwise WMaker claims $m_1x$ and that completes the argument). 
In the next round suppose that WBreaker claims $xy_1$. If $y_1 \in V(M)$, WMaker moves from $c$ to some $y\in U$, in round $i+4$, and then claims $yx$ in round $i+5$, which completes the analysis (at that moment $d_B(x)=4$). \\
If $y_1 \in U$, WMaker moves from $c$ to $y_1$ and then she claims $y_1m_2$ for some $m_2 \in U$ in round $i+5$. 
If WMaker is not able to visit $x$ in round $i+6$, this means one of the following:
\item[b.1)] WBreaker returned to $x$ along the edge $y_1x$ in round $i+5$ and then he claimed $xm_2$ in round $i+6$. So, WMaker needs to move to some $m_3 \in U$ in round $i+6$ and in the following round she is able to visit $x$ by claiming the edge $m_3x$, where $d_B(x)=5$. WBreaker is not able to prevent WMaker from visiting $x$ in round $i+7$ since he finished his previous move at vertex $m_2$.

\item[b.2)] WBreaker claimed $y_1y_2$ and $y_2y_3$, for some $y_2, y_3 \in U$, in rounds $i+5$ and $i+6$, respectively. Since $y_2y_3$ is not incident with $V(M)$, WMaker needs to move from $m_2$ to $y_2$ or $y_3$ in round $i+6$. Since $d_B(y_2)>d_B(y_3)$, she moves to $y_2$, as it is illustrated in Figure~\ref{fig1}.

\begin{figure}[h]

\centering
\includegraphics[width=0.4\textwidth]{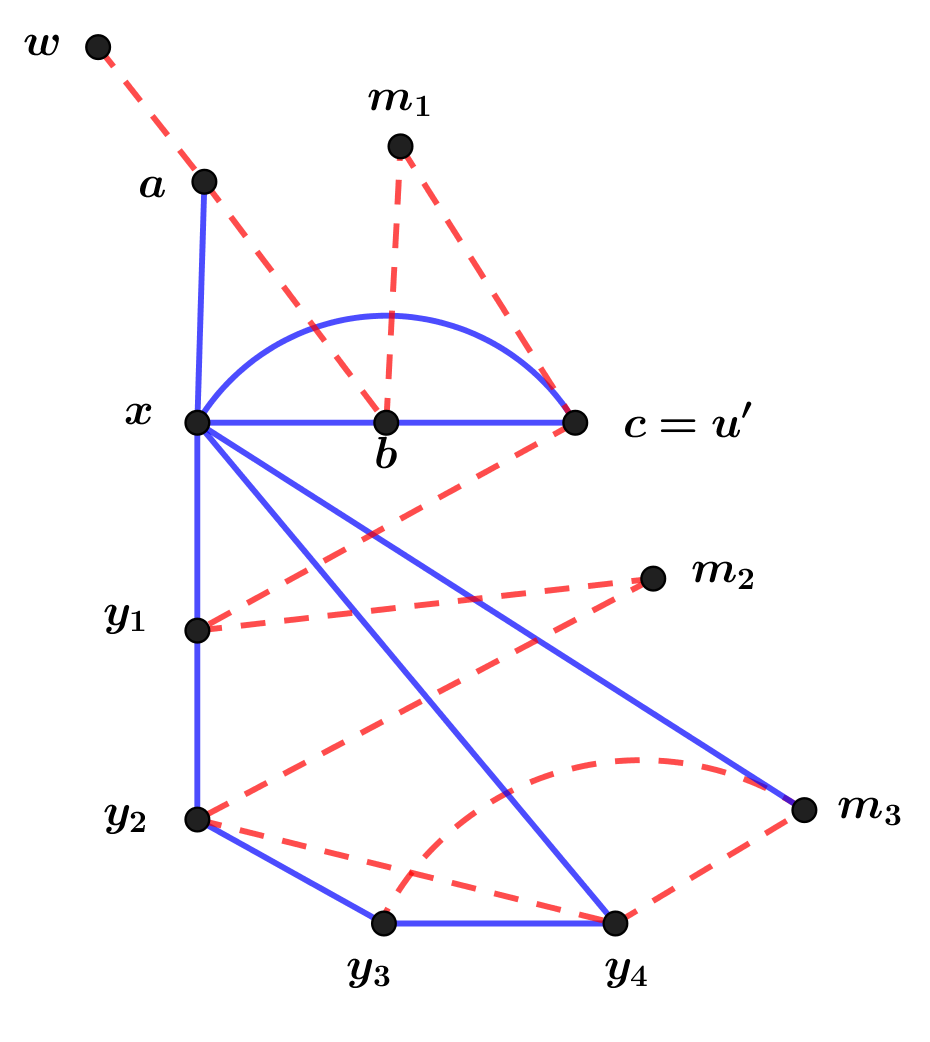}
\caption{WMaker's and WBreaker's moves in Case 2.i.b.2). when $c=u'$ }
 \label{fig1}
\small Red dashed lines show WMaker's moves and solid blue lines represent WBreaker's moves.
\end{figure}

If WMaker is not able to visit $x$ in round $i+7$, this means that there is, again, an edge in $E(B)$ not incident with $V(M)$. That is, WBreaker claimed $y_3y_4$ for some $y_4 \in U$. So, WMaker moves to $y_4$. In round $i+8$, WBreaker can make degree 5 at vertex $x$ by claiming $y_4x$ and in this way he prevents WMaker from visiting $x$. Then, WMaker claims $y_4m_3$ for some $m_3\in U$. If she is not able to move to $x$ in round $i+9$ because WBreaker claimed $xm_3$  (and at this moment $d_B(x)=6$), WMaker moves to $y_3$ and in the following round, $i+10$, she claims $y_3x$. WBreaker is not able to prevent WMaker from visiting $x$, because he finished his move at vertex $m_3$ in round $i+9$. 
\end{enumerate}

\item[ii.] WBreaker claims $bc$ and $cy_1$, for some $y_1\in U$, in rounds $i+2$ and $i+3$, respectively. In this case, WMaker first moves from $b$ to $u \in U$. Then, according to the strategy $\cS$, she needs to move to $c$, since $cy_1$ is not incident with $V(M)$ and $d_B(c)>d_B(y_1)$. If $b\neq u'$ and $c\neq u'$, then $u=u'$. If $b=u'$, WMaker already visited $b$. Also, if $c=u'$, WMaker already visited $c$.  \\
If in the next round WBreaker claims $y_1v$, for some $v \notin U$,  WMaker will be able to claim $cx$. Otherwise, if WBreaker claims edge $y_1y_2$ for some $y_2\in U$, in round $i+4$, WMaker needs to move from $c$ to vertex $y_2$. If afterwards WBreaker claims $y_2v'$ for some $v'\neq x$, then WMaker visits $x$ along the edge $y_2x$. \\
Otherwise, if WBreaker moves from $y_2$ to $x$ and $d_B(x)=3$, then WMaker must move to some $m\in U$. In round $i+6$, it is possible that $d_B(x)=4$ if WBreaker claims either the edge $xm$ or $xy$, for some vertex $y\neq m$. If WBreaker claims $xm$, WMaker claims $my_1$  and in the following round, $i+7$, WMaker claims $y_1x$. Whatever WBreaker plays in round $i+7$, he will not be able to prevent WMaker from visiting $x$. Otherwise, if WBreaker claimed $xy, y\neq m$, in round $i+6$, WMaker visits $x$ in this round. 
\item[iii.] If WBreaker returned to $x$ along the edge $bx$ in round $i+2$, then WMaker moves to some $m_1\in U$ along the edge $bm_1$ (if $b\neq u'$ WMaker moves to $m_1 = u'$ ). 
In round $i+3$, WBreaker can claim $xm_1$ and then $d_B(x)=3$.   In this case, WMaker must move to some $m_2 \in U$ ($m_2m_1 \notin E(B)$, otherwise it is a contradiction to Lemma \ref{lema1} after round $i+1$). Whatever WBreaker plays in round $i+4$, he will not be able to prevent WMaker from moving to vertex $x$ along the edge $m_2x$, because WBreaker finished his move in round $i+3$ at the vertex $m_1$. 
\end{enumerate}
\paragraph{Case 3.} According to Corollary~\ref{c1}, $d_B(w,U)\leq 1$ after round $i-1$. Since WBreaker claimed $ax$ in round $i$, WMaker is able to move from $w$ to some $u_i$, say $u_1$, and to $x$, because $d_B(x)=1$ and $x$ is adjacent only to $a$ in $B$. If $d_B(u_1) > d_B(x)$, she needs to move to $u_1$. 
The edges $xu_1, xu_2, u_1u_2 \notin E(B)$ due to Lemma~\ref{lema1}. Otherwise, she moves to $x$ which completes the analysis.
\\
In round $i+1$, suppose that WBreaker claims $xb$. If $b=u_2$, WMaker moves to $b=u_2$ since $d_B(b) > d_B(x)$.  The further analysis is similar to Case 1 so we skip details.  \\
If $b = u_1$, then WMaker moves to $u_2$ and in the following round, $i+2$, she moves to $x$. \\
If $b \in U$ and  $b \notin \{u_1,u_2\}$, then WMaker needs to move from $u_1$ to $x$ or $b$ because $xb$ is not incident with $V(M)$. Since $d_B(x)=2$ and $d_B(b)=1$, she moves to $x$ (at that moment $d_B(x)=2$).   \\
If $b \in V(M)$, WMaker first moves to $u_2$, if $d_B(u_2)\geq d_B(x)$ (suppose that she moves to $u_2$ even if $d_B(u_2)=d_B(x)$), and then claims $u_2x$. Otherwise, she visits $x$ in round $i+1$ along the edge $u_1x$. In both cases at the moment when WMaker visits $x$, $d_B(x)=2$.
\end{proof}
\subsection{Proof of Theorem \ref{teorema1}}
\begin{proof}
We are going to describe WMaker's strategy and prove that she can follow this strategy. At the beginning of the game, all vertices are isolated in WMaker's graph and $U=V(K_n)$.

\paragraph{Stage 1.} In this stage WMaker builds a path $P$ of length $n-4$ in $n-4$ rounds, by playing according to the strategy $\cS$, which is possible due to Corollary~\ref{c2}. 
\paragraph{Stage 2.} During the course of this stage WMaker visits the three remaining vertices in at most 5 additional moves. At the beginning of this stage suppose that WMaker is at vertex $w$  and $U= \{ u_1, u_2, u_3 \}$. Assume that it is WMaker's turn to play her move in round $n-3$.  Corollary~\ref{cc1} implies that $d_B(w,U)\leq 2$.\\
First, suppose 
that vertex $w$ is such that $d_B(w,U)\leq 1$. Let $wu_1 \in E(B)$. 
Since WMaker visited $w$ in round $n-4$ this means that WBreaker must have claimed $wu_1$ in this round. Otherwise, if WBreaker claimed this edge earlier, then we would have a contradiction to Lemma \ref{lema1} before round $n-4$. If WBreaker finished his move in round $n-4$ at vertex $u_1$, then in his $(n-3)^{\mathrm{rd}}$ move he could claim $u_1u_i$, for some $i \in \{2,3 \}$. Suppose that $u_1u_2 \in E(B)$ and WBreaker is at $u_2$. This edge could not exist earlier, because it would be a contradiction to Lemma \ref{lema1}. Also, the edges $u_2u_3, u_1u_3 \notin E(B)$ because of Lemma \ref{lema1}.
In round $n-3$, WMaker claims the edge $wu_3$ and in the following round she moves to $u_1$. WBreaker is not able to prevent WMaker from claiming $u_3u_1$ because he finished his $(n-3)^{\mathrm{rd}}$ move at the vertex $u_2$. \\
In round $n-3$, when WMaker visited $u_3$ for the first time, we had $d_B(u_3) \leq 6$ (according to Lemma~\ref{lema3}). Since $u_1,u_2$ were still unvisited by WMaker in round $n-3$, we have $d_B(u_1), d_B(u_2) \leq 8$, in round $n-1$.
Since $d_B(u_1,V(P)) + d_B(u_2,V(P)) < v(P) $, there exists a vertex $v\in V(P)$ such that $u_1v$ and $vu_2$ are free. She claims $u_1v$ in round $n-1$. If WMaker is not able to claim $vu_2$ in round $n$, this means that WBreaker finished his move in the previous round at vertex $u_2$, so he was able to prevent WMaker from visiting $u_2$ by claiming $u_2v$ in his $n^{\mathrm{th}}$ move. 
Then WMaker moves to some $v' \in V(P)$ such that edges $vv'$ and $v'u_2$ are free. We need to prove that such vertex $v'$ exists. \\
Let $P' = P \setminus \{v, u_1 \}$. \\
Since $d_B(u_2) \leq 6$ in round $n-3$, we have $d_B(u_2) \leq 8$ before WMaker's move in round $n$. 
So, if $d_B(v,V(P')) + d_B(u_2,V(P')) \geq v(P') = n-3$, 
it follows that $d_B(v,V(P')) \geq n-11$. To make such a large degree at vertex $v$, WBreaker needed at least $2(n-11)-2$ moves because he is also a walker. Since he played exactly $n$ moves, this is not possible. \\
So, in round $n$ WMaker claims $vv'$ and, in the last round, $n+1$, she moves to $u_2$. WBreaker is not able to prevent WMaker from claiming $v'u_2$, because he finished his move in round $n$ at vertex $v$. \\ \\
Let $d_B(w,U) =2$ before WMaker's move in round $n-3$ and let $wu_1, wu_2 \in E(B)$. 
From Corollary~\ref{c1} we know that WBreaker has moved to $u_1$ or $u_2$ from vertex $w$ in his last move, because at the end of round $n-4$ when WMaker came to $w$, we had $d_B(w,U)\leq 1$. Assume that WBreaker is at vertex $u_2$.
Edges $u_1u_2, u_2u_3, u_1u_3 \notin E(B)$. Otherwise, this would mean that WBreaker claimed some of these edges in some round before round $n-3$ and we would have a contradiction to 
Lemma \ref{lema1}. \\
WMaker claims $wu_3$ in round $n-3$.
If in the following round WBreaker moves to $u_3$, WMaker claims $u_3u_1$ and then $u_1u_2$ in round $n-1$. WBreaker is not able to prevent WMaker from claiming $u_1u_2$ because he finished his move in round $n-2$ at vertex $u_3$.  \\
If WBreaker moves to $u_1$ in round $n-2$, WMaker claims $u_3u_2$. In round $n-1$ WMaker identifies a vertex $v\in V(P) $ such that edges $u_2v$ and $vu_1$ are free. 
In the similar way as above we can prove that such vertex $v$ exists. WMaker claims $u_2v$ in round $n-1$ and in round $n$, she claims $vu_1$. 
Since WBreaker must move from $u_1$ in round $n-1$, he can not prevent WMaker from claiming $vu_1$ in the last round $n$.  Otherwise, WMaker can visit the remaining two vertices in two moves. 
\end{proof}
\subsection{The proof of Theorem \ref{teorema2}}
\begin{proof}
First, we describe WMaker's strategy and then prove that she can follow it. At the beginning of the game $U=V(K_n)$.
\paragraph{Stage 1.} In the first $n-4$ rounds WMaker builds a path $P$ of length $n-4$ (with $n-3$ vertices) by playing according to the strategy $\cS$. 
\paragraph{Stage 2.} In the next at most 4 rounds, WMaker closes the cycle of length $n-2$ or $n-1$. \\
Denote by $v_1$ the vertex in which WMaker starts the game. 
In round $n-3$, WMaker from her current position moves to vertex $u_i \in U$, $i\in\{1,2,3\}$ which is not incident with $v_1$ in $B$. 
If WMaker is able to claim the edge $u_iv_1$ in her following move, then she claims it and creates a cycle of length $n-2$.
Otherwise, she moves to $u_j$ along the edge $u_iu_j$, where $i,j\in \{1,2,3 \}$ and $i \neq j$, in round $n-2$. \\
If the edge $u_jv_1$ is free after WBreaker's move in round $n-1$, WMaker claims this edge and closes the cycle of length $n-1$. \\
Otherwise, she finds a vertex $v \in V(P)$ such that edges $u_jv$ and $vv_1$ are free. WMaker first claims the edge $u_jv$ and then in the following round, $n$, she claims $vv_1$ and thus closes the cycle of length $n-1$. 
\paragraph{Stage 3.} Depending on how Stage 2 ended, WMaker completes the Hamilton cycle in at most 8 rounds. We give the details later.\\ \\
We now prove that WMaker can follow her strategy.
\paragraph{Stage 1.} Corollary \ref{c2} implies that WMaker can follow her strategy in Stage 1 and build the path $P$ of length $n-4$, thus visiting $n-3$ vertices in $n-4$ moves. 
\paragraph{Stage 2.} At the beginning of round $n-3$, WMaker is positioned at vertex $x$ and $U= \{u_1,u_2,u_3 \}$. We know that WMaker started the game at the vertex $v_1$ and we consider several cases. 
\paragraph{Case 1.} WBreaker is not positioned at vertex $v_1$ at the beginning of WMaker's move in round $n-3$. 
\paragraph{Case 1.a.} Suppose that $d_B(x,U)=2$ before WMaker's $(n-3)^{\mathrm{rd}}$ move. Let $u_1x, u_2x \in E(B)$. 
From Corollary \ref{cc1} we know that $u_3x \notin E(B)$, because after WBreaker's move (and before WMaker's move) in each round we can have $d_B(x,U)\leq 2$. Also, WBreaker must be positioned at $u_1$ or $u_2$, that is, one of the edges, $u_1x$ or $u_2x$ is the edge which WBreaker claimed in his last move. Otherwise, we will have a contradiction to Corollary~\ref{c1} and Corollary~\ref{cc1}. Suppose that WBreaker finished his $(n-3)^{\mathrm{rd}}$ move at vertex $u_1$. 
\begin{claim} \label{claim1}
For all $i\in\{1,2,3\}$, $v_1u_i \notin E(B)$, before WMaker's $(n-3)^{\mathrm{rd}}$ move.
\end{claim}
\begin{proof}
Suppose that at least one of these three edges is in WBreaker's graph. \\
Based on the above consideration, we know that WBreaker claimed the edge $u_2x$ in his $(n-4)^{\mathrm{th}}$ move, and the edge $xu_1$ in his $(n-3)^{\mathrm{rd}}$ move. In round $n-5$ he moved to $u_2$ from some vertex. Assume that this vertex is $v_1$. So, $u_2v_1\in E(B)$. This means that in round $n-6$ he moved from some vertex of his path (because he is also a walker) to vertex $v_1$. Since in this round he did not visit any new vertex from $U$, Lemma \ref{lema2} implies that there can be at most 2 vertices from $U$ visited by WBreaker. Denote them by $u$ and $u'$. \\
If WMaker cannot visit any of these vertices in her $(n-6)^{\mathrm{th}}$ move from her current position, denoted by $y$, this means that $d_B(y,U)=2$ which implies that WBreaker claimed one of the edges $yu$ or $yu'$ in round $n-6$ (Corollary \ref{cc1}), and the other one in round $n-7$. A contradiction to assumption.  \\
Thus, WMaker is able to visit, say $u$ from $y$ in round $n-6$ and in the following round vertex $u'$ from $u$. The edge $uu' \notin E(B)$, as Lemma~\ref{lema1} holds before round $n-6$. 
If $d_B(x)=1$, in round $n-4$ strategy $\cS$ tells WMaker to claim $u'u_2$ ($u'u_2 \notin E(B)$ as it would contradict Lemma~\ref{lema1} after round $n-6$), because $d_B(u_2)=2>d_B(x)$. A contradiction, because WMaker visited $x$ in round $n-4$. \\
If $d_B(x)\geq 2$ before WMaker's move in round $n-4$, this means that there exists at least one vertex $a$ such that $xa, xu_2 \in E(B)$ and $a$ is visited by WMaker. It follows that WBreaker claimed $ax$ in some round before round $n-6$. This implies that $x=u$ or $x=u'$ at the beginning of WMaker's move in round $n-6$. Therefore WMaker visited $x$ in round $n-6$ or $n-5$. A contradiction. \\
If $u_1v_1 \in E(B)$, then it follows that WBreaker could claim it, at  latest, in round $n-6$. By similar consideration as above, we can conclude that this is also not possible. \\
If $v_1u_3 \in E(B)$, then we can consider the following cases:

\begin{enumerate}
\item WBreaker in some round $k\leq n-7$ moved from $u_3$ to $v_1$, where $d_B(u_3)\geq 2$, since WBreaker is a walker. After his move, there can be at most one more vertex $u'\in U$ visited by WBreaker (Lemma \ref{lema2}). Then, WMaker from her current position, say $y$, can visit $u'$ or $u_3$, according to Corollary \ref{cc1}.
Suppose that WMaker moved to $u'$. 
In round $k+1$, strategy $\cS$ will tell WMaker to move to $u_3$. (WBreaker could move from $v_1$ to some $u''\in U$, but $d_B(u_3)>d_B(u'')$). A contradiction.   
\item WBreaker in some round $k\leq n-8$ moved from his current position $p$ to vertex $v_1$ and then in round $k+1$ he claimed $v_1u_3$. After WBreaker's move in round $k$ there can be at most two vertices from $U$, say $u,u'$, visited by WBreaker (Lemma \ref{lema2}). From Corollary \ref{cc1} and Lemma \ref{lema1} it follows that WMaker can visit both vertices in rounds $k$ and $k+1$. Suppose that she first moves to $u$ and then claims the edge $uu'$. If she is not able to visit $u_3$ in round $k+2$, this means that WBreaker claimed $u_3u'$ in round $k+2$. Then WMaker claims $u'u''$, for some $u''\in U$ and then $u''u_3$, as after WBreaker's move in round $k+4$ there is no vertex from $U$ which has larger degree than $u_3$ and there is no edge whose both endpoints are in $U$. A contradiction.
\end{enumerate}
\end{proof}
\noindent Claim ~\ref{claim1} gives that in her $(n-3)^{\mathrm{rd}}$ move, WMaker can claim the edge $xu_3$ and in the following round, $n-2$, she can close a cycle of length $n-2$, by claiming the edge $u_3v_1$. 
\paragraph{Case 1.b.} Suppose that $d_B(x,U)=1$ before WMaker's move in round $n-3$. Let $xu_1 \in E(B)$. Since WMaker visited $x$ in round $n-4$, this means that WBreaker claimed $xu_1$ in round $n-4$, as otherwise this would contradict Lemma \ref{lema1} before round $n-4$. Assume that WBreaker finished his $(n-4)^{\mathrm{th}}$ move in $u_1$. Otherwise, if he finished $(n-4)^{\mathrm{th}}$ move in $x$, after his move in round $n-3$ we can have Case 1.a. which we already considered, or WBreaker moved to some $v$ on WMaker's path $v\neq v_1$. 
\begin{claim} \label{claim2}
After round $n-4$, $u_iv_1 \notin E(B)$ for each $i \in \{1,2,3 \}$. 
\end{claim}
The proof of this Claim is very similar to the proof of Claim~\ref{claim1}, therefore we give it in the appendix. 

If WBreaker, in his $(n-3)^{\mathrm{rd}}$ move, claims $u_1u_2$ (or $u_1u_3$), then WMaker moves from $x$ to $u_3$ (respectively to $u_2$). In the following round, $n-2$, she claims $u_3v_1$ (or $u_2v_1$), which is free according to Claim~\ref{claim2}, and closes the cycle of length $n-2$.\\ 
If WBreaker, in his $(n-3)^{\mathrm{rd}}$ move, claims $u_1v_1$, then WMaker moves to $u_2$ (or $u_3$). Suppose WMaker claimed $xu_2$. In the following round WBreaker can claim $v_1u_2$. So, WMaker is not able to close the cycle of length $n-2$ in round $n-2$. In that case, she moves to $u_3$ along $u_2u_3$ ($u_2u_3 \notin E(B)$ as this contradicts Lemma \ref{lema1} before round $n-3$). In round $n-1$, WMaker moves from $u_3$ to $v_1$ and makes a cycle of length $n-1$. WBreaker cannot block her because he finished his $(n-3)^{\mathrm{rd}}$ move at vertex $u_2$. The analysis is the same if WBreaker claimed $xu_3$ in round $n-2$.
\paragraph{Case 1.c.} Suppose that $d_B(x,U)=0$ before WMaker's move in round $n-3$. 
\begin{claim} \label{claim3}
Before WMaker's $(n-3)^{\mathrm{rd}}$ move, there can be at most two vertices from $U$ adjacent to $v_1$ in $B$.
\end{claim}
\begin{proof}
If $u_1v_1,u_2v_1, u_3v_1 \in E(B)$, then WBreaker spent $4$ moves to do this. Assume that he claimed edges in this order: $u_1v_1,v_1u_2, u_2v_1$ and $v_1u_3$. Thus, he moved from $u_1$ to $v_1$ in round $n-6$  which means that after his move in this round, there could be at least one more vertex $u$ from $U$ touched by WBreaker (Lemma \ref{lema2}), since he did not visited any new vertex from $U$ in this round. WMaker can visit $u$ or $u_1$ in round $n-6$ from her current position, say vertex $y$. Otherwise, we would have $d_B(y,U)=2$ and $yu_1,yu \in E(B)$ (Corollary \ref{cc1}) and this would mean that WBreaker moved from $y$ in round $n-6$, which is not the case. Assume that WMaker visited $u$ in this round. After WBreaker claims $v_1u_2$, in round $n-5$, WMaker would be able to move from $u$ to $u_1$ or $u_2$ in this round (both edges $uu_1$ and $uu_2$ are free in the moment WMaker wants to claim them, due to Lemma~\ref{lema1}). A contradiction. 
\end{proof}
Claim~\ref{claim3} implies that there can be at most two vertices from $U$, adjacent to $v_1$ in $B$. If there are exactly 2 vertices, say $u_1,u_2$ adjacent to $v_1$, then in the similar way as in the proof of Claim~\ref{claim3}, we can show that WBreaker finished his last move, in round $n-3$, in vertex $u_1$ or $u_2$. \\ \\
In her $(n-3)^{\mathrm{rd}}$ move, WMaker claims $xu_3$. Whatever WBreaker plays in round $n-2$, he will not be able to prevent WMaker from claiming $u_3v_1$ in this round. 
Thus, WMaker closes the cycle of length $n-2$ in round $n-2$. 
\paragraph{Case 2.} Suppose that WBreaker is at vertex $v_1$ after his move in round $n-3$. 
\begin{claim} \label{claim4}
It is not possible that at the same time $u_iv_1, u_jv_1 \in E(B)$ for some $i,j \in \{1,2,3\}$, $i\neq j$ and that WBreaker is at $v_1$ after his move in round $n-3$. 
\end{claim}
\begin{proof} It the assertion of the claim was true, it would mean that WBreaker spent three moves claiming edges in the following order: $u_iv_1$, $v_1u_j$ and $u_jv_1$, for some $i,j\in \{1,2,3 \},\linebreak  i \neq j$. This is not possible because of the following. Let  $i=1$ and $j=2$.\\
WBreaker came to vertex $u_1$ in round $n-6$ from some vertex $p$ which was on his path because WBreaker is a walker. 
\begin{enumerate}
\item[a)] If WMaker was at $p$ in that moment, then after WBreaker claimed $pu_1$ we can have $d_B(p,U) \leq 2$ (Corollary \ref{cc1}). Suppose $d_B(p, U)=2$. Let $pu_1,pu' \in E(B)$ for some $u'\in U$. This means that in her move in round $n-6$, WMaker must move to some $u\in U$.
After WBreaker claimed $u_1v_1$ in round $n-5$, edge $uu'$ was free and WMaker could claim it. Then, in round $n-4$ WBreaker claims $v_1u_2$ and WMaker according to the strategy $\cS$ must visit $u_1$ or $u_2$ from $u'$.
Since $d_B(u_1)>d_B(u_2)$ she will visit $u_1$. A contradiction. \\
Edges $pu$, $uu'$ and $u'u_1$ (or $u'u_2$) must be free in the moment when WMaker wants to claim them. Otherwise, this would be in contradiction to Lemma \ref{lema1}. Similarly, we can obtain a contradiction in case $d_B(p,U)= 1$ and WMaker is at $p$ after WBreaker's move $pu_1$.
\item[b)] Suppose that WMaker is at some vertex $w \neq p$ and $p,u_1 \notin V(M)$. Due to Lemma \ref{lema1}, WMaker must visit either $p$ or $u_1$ in her $(n-6)^{\mathrm{th}}$ move. Assume she visits $p$ along $wp$. Let $u' \in U$ be another vertex touched by WBreaker (because after each round there can be at most $2$ vertices in $U$ belonging to $V(B)$ - Lemma \ref{lema2}). 
After WBreaker claims $u_1v_1$ in round $n-5$, WMaker can claim $pu'$ and when WBreaker moves to $u_2$ along $v_1u_2$  WMaker visits $u_1$ from $u'$ in round $n-4$. A contradiction. \\
Note that edges $pu'$ and $u'u_1$ (or $u'u_2$) must be free in the moment when WMaker wants to claim them. Otherwise, we would have a contradiction to Lemma \ref{lema1}. 
\item[c)] Suppose that WMaker is at some vertex $w$  in the moment when WBreaker claimed $pu_1$ and let $p \in V(M)$. After WBreaker's move in round $n-6$ there can be at most two more vertices from $U$ (beside the vertex $u_1$), say $a$ and $b$, which belong to $V(B)$ (due to Lemma \ref{lema2}).  Then, WMaker can move from $w$ to $a$ or $b$ because $d_B(w,U)\leq 1$ (due to Corollary \ref{c1} and because WBreaker in his last move chose $pu_1$). Suppose that $wa$ is free and WMaker claims it. After WBreaker claims $u_1v_1$ in round $n-5$, WMaker chooses $ab$ (this edge must be free, otherwise we would have a contradiction to Lemma \ref{lema1} after round $n-7$). In round $n-4$ WBreaker claims $v_1u_2$ and WMaker is able to claim $bu_1$ (otherwise we have a contradiction to Lemma \ref{lema1}, again). A contradiction, because WMaker visited $x$ in round $n-4$. \\
Therefore, it is not possible that at the same time $u_iv_1, u_jv_1 \in E(B)$ for some $i,j\in\{1,2,3\}$, $i\neq j$.
\end{enumerate}
\end{proof}
\noindent Claim~\ref{claim4} implies that there can be at most one edge $v_1u_i \in E(B)$ for $i\in\{1,2,3\}$. Thus, suppose that edge $v_1u_2 \in E(B)$, that is, WBreaker came from $u_2$ to $v_1$. WMaker is at vertex $x$. We know that $d_B(x,U) \leq 1$ because of Corollary \ref{c1} and 
because in his last move WBreaker moved to $v_1$. Suppose the edge $xu_1$ is free. WMaker claims it. If edge $v_1u_1$ is free at the beginning of WMaker's $(n-2)^{\mathrm{nd}}$ move, then WMaker claims it and closes the cycle of length $n-2$. Otherwise, this means that WBreaker claimed this edge in his $(n-2)^{\mathrm{nd}}$ move (Claim \ref{claim4}). In this case, WMaker moves to $u_3$. The edge $u_1u_3$ must be free due to Lemma \ref{lema1} after round $n-4$. In round $n-1$ WBreaker is not able to prevent WMaker from claiming the edge $u_3v_1$ because he finished his previous move at vertex $u_1$. So, the cycle of length $n-1$ is created in WMaker's graph. \\
Therefore, WMaker is able to create a cycle of length $n-2$ or $n-1$.
\paragraph{Stage 3.} Depending on how Stage 2 ended we analyse two cases. 
\paragraph{Case 1.} Suppose that WMaker created a cycle $C$ of length $n-2$. She played exactly $n-2$ rounds. WMaker's current position is at vertex $v_1$. Denote by $v_{n-2}$ the vertex which was last visited by WMaker in round $n-3$. Let $U = \{u_1,u_2 \}$. In round $n-1$ WMaker returns from $v_1$ to vertex $v_{n-2}$.
\begin{figure}[h]

\centering
\includegraphics[width=0.8\textwidth]{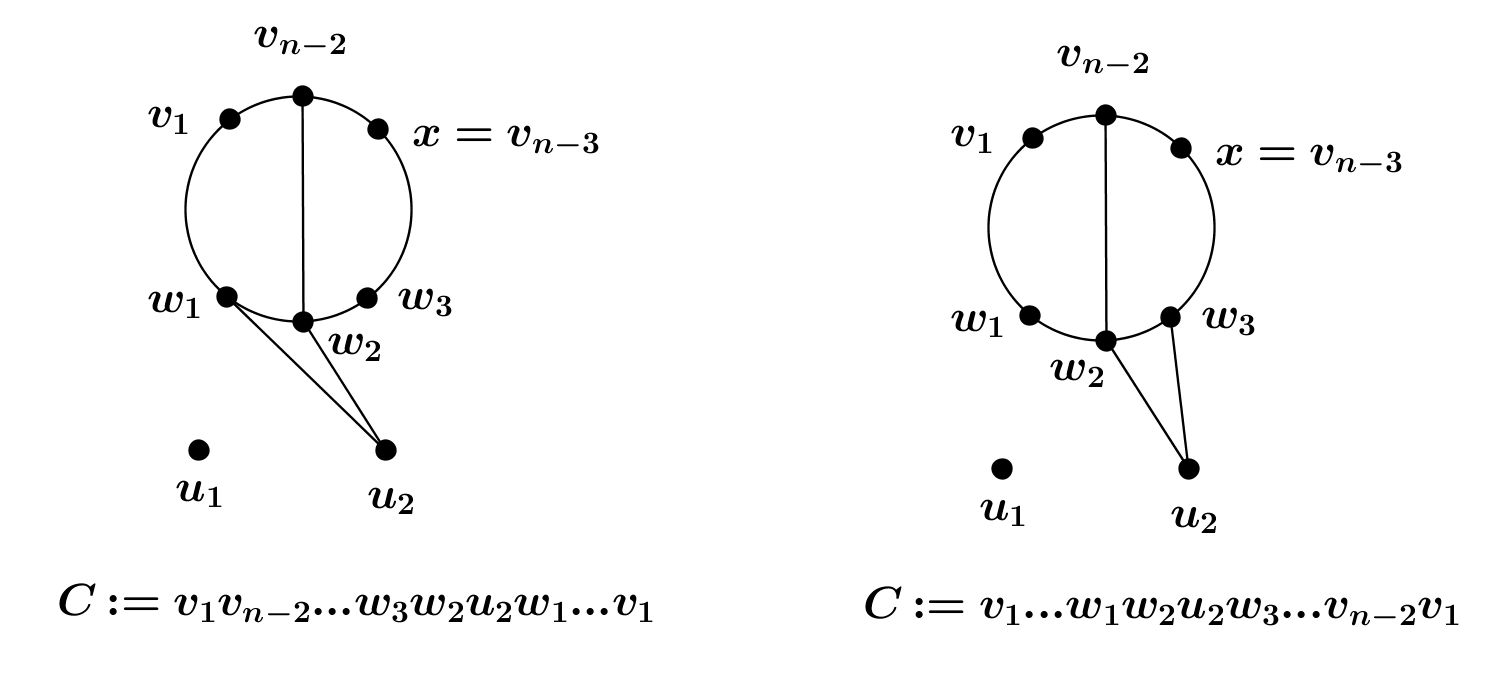}
\caption{WMaker's cycle $C$ of the length $n-1$ after round $n+2$.}
\label{fig2}
\small Left (right) figure illustrates the case when WMaker claimed $u_2w_1$ ($u_2w_3$) in round $n+2$.
\end{figure} 
Lemma~\ref{lema3} guarantees that $d_B(u_i)\leq 6$ for all $i\in \{1,2,3\}$ in round $n-3$ (in which WMaker visited $v_{n-2}$) and so in that moment $d_B(u_i,C) \leq 6$ for $i\in\{1,2\}$. It follows that, after WBreaker's move in round $n$, vertices $u_1,u_2,v_{n-2}$ can have degree at most $8$ in $B$ towards cycle $C$. Since $v(C)=n-2$, by pigeonhole principle, there are 3 consecutive vertices $w_1,w_2,w_3$ on cycle $C$ such that there are no WBreaker's edges between $\{w_1,w_2,w_3 \}$ and $\{ u_1,u_2,v_{n-2} \}$.
Since $n$ is large enough, there are at least $4$ such triples (with at least 6 vertices if triples are not disjoint). If one of these 6 vertices, say $t$, is such that $d_B(t)\geq n/3$, WMaker can pick another triple not containing such a vertex. There can be at most one vertex with degree at least $n/3$. Otherwise, this would mean that WBreaker played more than $n$ moves (as he is a walker), which is a contradiction. 
\\
WMaker first claims the edge $v_{n-2}w_2$ in round $n$. If in the following round, WBreaker claimed the edge $w_iu_1$ for some $i\in\{1,2,3\}$, then WMaker chooses edge $w_2u_2$ (otherwise, if he claimed $w_iu_2$, we just interchange the vertices $u_1$ and $u_2$) and in round $n+2$ closes the cycle $C$ of length $n-1$ by claiming edge $u_2w_1$ or $u_2w_3$. This is illustrated on Figure~\ref{fig2}. 
If WBreaker did not claim any of edges $w_iu_1, w_iu_2$, for $i\in\{1,2,3\}$, in round $n+1$, then WMaker moves from $w_2$ to either of these two vertices $u_1, u_2$. In round $n+2$, she moves from chosen vertex ($u_1$ or $u_2$) to $w_1$ or $w_2$, because WBreaker could not claim both edges $u_iw_1$ and $u_iw_2$ in round $n+2$, where $u_i$, $i\in\{1,2\}$, is the vertex which WMaker chose in the previous round. 
 \\
Let $U=\{u_1\}$. Suppose that WMaker finished her last move at $w_1$. Consider the following cases. 
\begin{enumerate}
\item[a)] WBreaker finished his $(n+3)^{\mathrm{rd}}$ move at vertex $u_1$. Then WMaker returns from $w_1$ to vertex $u_2$ which now belongs to $C$. Similarly as above we conclude that there are three vertices $y_1,y_2,y_3$ on $C$ such that there are no Breaker's edges between  $\{y_1,y_2,y_3 \}$ and $\{ u_1,u_2 \}$. WMaker in round $n+4$ claims $u_2y_2$ and then in rounds $n+5$ and $n+6$ claims $y_2u_1$, $u_1y_1$ (or $u_1y_3$), respectively. 
\item[b)]WBreaker finished his $(n+3)^{\mathrm{rd}}$ move at some vertex $v \neq u_1$. Then WMaker from $w_1$ finds three vertices $y_1,y_2,y_3$ such that there are no Breaker's edges between  $\{y_1,y_2,y_3 \}$ and $\{ w_1,u_1 \}$. She first claims $w_1y_2$ and then $y_2u_1$ and $u_1y_1$ (or $u_1y_3$), and completes the Hamilton cycle in round $n+5$. 
\end{enumerate} 
\paragraph{Case 2.} Suppose that in Stage $2$ WMaker created a cycle $C$ of length $n-1$ (in at most $n$ rounds). Denote by $v_{n-1}$ the vertex from $U$ that was last visited by WMaker. In the last round of Stage $2$, WMaker moved from $v_{n-1}$ to $v_1$. Let $U = \{u\}$. \\
WMaker first moves from $v_1$ to $v_{n-1}$ in round $k\leq n+1$ ($k=n+1$ if WMaker finished her cycle in round $n$ of previous stage). \\ 
If WBreaker finished his move in round $k+1\leq n+2$ at some vertex different from $u$, then WMaker finds three consecutive vertices $ y_1,y_2,y_3$ on $C$, such that there are no edges between $\{y_1,y_2,y_3 \}$ and $\{ u, v_{n-1}\}$ in $B$. Since $d_B(v_{n-1}), d_B(u) \leq 9$ in round $k+1 \leq n+2$ and since $v(C) = n-1$, by pigeonhole principle, such vertices $y_1,y_2,y_3$ exist. She first claims $v_{n-1}y_2$ and in the following round the edge $y_2u$. In the last round, $k+3\leq n+4$, she claims $uy_1$ or $uy_3$ because WBreaker could not claim both edges. 
\\ 
If WBreaker finished his move in round $k+1\leq n+2$ at vertex $u$, then WMaker moves from $v_{n-1}$ to $v_{n-2}$ in this round. In round $k+2 \leq n+3$ WBreaker must move from $u$. In this round WMaker finds three consecutive vertices $y_1,y_2,y_3$ on $C$, such that there are no edges between $\{y_1,y_2,y_3 \}$ and $\{ u, v_{n-2}\}$ in $B$.  
These vertices exist as WMaker visited vertex $v_{n-2}$ in round $n-3$ and in that moment we had $d_B(v_{n-2}), d_B(u) \leq 6$ (Lemma \ref{lema3}). Since $d_B(v_{n-2}), d_B(u) \leq 9$ in round $k+2 \leq n+3$ and since $v(C) = n-1$, by pigeonhole principle, such vertices $y_1,y_2,y_3$ exist. WMaker first claims the edge $v_{n-2}y_2$. In the following round WMaker claims $y_2u$ and in the final round, $k + 4 \leq n+5$, she completes the Hamilton cycle by claiming edge $uy_1$ or $uy_3$.
\end{proof}

\section{Proof of Theorem~\ref{breaker}}
\label{brP}
In this section we prove Theorem~\ref{breaker}, thus providing WBreaker's strategy in the Connectivity game. Here, we suppose that WMaker starts the game and $U=V(K_n)$.
\begin{proof}
WBreaker plays arbitrarily until $|U|=3$. To be able to visit $n-3$ vertices, WMaker needs to play at least $n-4$ moves. Let $u_1,u_2,u_3 \in U$ after round $k\geq n-4$. \\
If in round $k+1\geq n-3$ WMaker moves to some vertex $v\neq u_i$, $i\in\{1,2,3\}$, then she will need at least 3 more moves to visit $u_1,u_2,u_3$, which satisfies the claim. \\
Suppose that in round $k+1\geq n-3$ WMaker moves to some $u_i$, $i\in \{1,2,3\}$. WBreaker moves to $u_j$, $j\neq i$. WBreaker is able to move to $u_j$ since $u_j \in U$ and there is no WMaker's edge between WBreaker's current position and vertex $u_j$. \\
Without loss of generality, suppose that WMaker has moved to $u_1$ and WBreaker to $u_2$. If WMaker in round $k+2$ moves to one of $\{u_2,u_3\}$, WBreaker claims the edge $u_2u_3$. As $u_2u_3 \in E(B)$ from her current position, WMaker is not able to visit the remaining isolated vertex in her graph in round $k+3\geq n-1$, so she needs to make at least one additional move to touch the remaining vertex.
If WMaker moves to some vertex $v \neq u_i$, $i \in \{ 2,3 \}$ in round $k+2\geq n-2$, then she will need at least two more moves to visit $u_2,u_3$. \\
It follows that WMaker needs at least $n$ moves to win in the Connectivity game. 
\end{proof}

\section{Concluding remarks}
\label{sec::last}

Theorems~\ref{teorema1} and~\ref{breaker} imply that WMaker needs $t$, $n\leq t \leq n+1$ moves to make a spanning tree and Theorem~\ref{teorema2} gives that she needs at most $n+6$ moves to create a Hamilton cycle. Similar reasoning to the proof of Theorem~\ref{breaker} leads to the conclusion that for creating a Hamilton cycle, WMaker needs at least $n+1$ moves, as she cannot make a spanning tree in less than $n$ moves. 

Note that if we increase WBreaker's bias $b$ only by one, in the $(1:2)$ WMaker--WBreaker game, WMaker will not be able to visit all vertices of the graph.
This is because WBreaker can isolate a vertex in WMaker's graph. Indeed, if $b=2$, in each round he uses one move to return to the fixed vertex along previously claimed edge, and the other to claim the edge between this particular vertex and WMaker's current position.

\section*{Acknowledgements}
The research of the second author is partly supported by Ministry of Education, Science and Technological Development, Republic of Serbia, Grant No.\ 174019.

\appendix
\section{Appendix}
\begin{proof}[Proof of Claim~\ref{claim2}]
Suppose that $\exists i \in \{1,2,3 \}$ such that $v_1u_i \in E(B)$. Since WBreaker moved from $x$ to $u_1$ in round $n-4$, it follows that in round $n-5$, he came from some vertex $p$ to $x$. We know that $p \notin \{ u_2,u_3 \}$ following $\cS$ and due to Lemma~\ref{lema1}. Thus, it could happen that WBreaker claimed $u_2p$ (or $u_3p$) in round $n-6$ and  in round before that, he claimed $v_1u_2$ (or $v_1u_3$). Let $v_1u_2, u_2p \in E(B)$. This means that WBreaker returned to $v_1$ along some edge in round $n-8$. Since in this move WBreaker did not touch any new vertex from $U$, after his move in round $n-8$, there can be at most 2 vertices from $U$ visited by WBreaker (according to Lemma \ref{lema2}), say $t$ and $t_1$. In rounds $n-8$ and $n-7$, WMaker visits $t$ and $t_1$, respectively. (If WMaker cannot move from current position, say $y$, to $t$ or $t_1$ in round $n-8$, then Corollary \ref{cc1} implies that $d_B(y,U)=2$ and this would mean that WBreaker moved from $y$ to $t$ or $t_1$ in round $n-8$, which is not the case.) \\ \\
If WMaker is not able to claim $t_1u_2$ in round $n-6$ this means that WBreaker claimed $t_1u_2$ in this round (the edge could not appear earlier due of Lemma \ref{lema1}). In round $n-6$, WMaker moves from $t_1$ to some other vertex $u\in U$. In the following round, $n-5$, WBreaker will not be able to prevent WMaker from claiming $uu_2$ because he finished his $(n-6)^{\mathrm{th}}$ move in $t_1$. A contradiction. \\
Thus, after round $n-4$, all edges $u_1v_1, u_2v_1, u_3v_1$ are free. 
\end{proof}
\end{document}